\documentclass[letter, reqno, 14pt]{amsart}

\usepackage[usenames,dvipsnames]{color}
\usepackage{amsthm,amsfonts,amssymb,amsmath,amsxtra}
\usepackage{tikz}
\usepackage{verbatim}
\usepackage{amsmath}
\usepackage{amsxtra}
\usepackage{amscd}
\usepackage{amsthm}
\usepackage{amsfonts}
\usepackage{amssymb}
\usepackage{eucal}
\usetikzlibrary{arrows}
\usepackage[all]{xy}
\SelectTips{cm}{}
\usepackage{xr-hyper}
\usepackage[colorlinks=
   citecolor=Black,
   linkcolor=Red,
  urlcolor=Blue]{hyperref}
\usepackage{verbatim}

\usepackage[margin=1.25in]{geometry}
\usepackage{mathrsfs}

\makeatletter
\newcommand{\extraNode}[6]%
{%
\dynkinPlaceRootRelativeTo{#1}{#2}{#3}{#4}{#5}
\dynkinIndefiniteSingleEdge{#1}{#2}
\dynkinRootMark{o}{#1}
\advance\dynkin@nodes by 1
\dynkinLabelRoot{#1}{#6} 
}%
\makeatother

\RequirePackage{xspace}
\RequirePackage{etoolbox}
\RequirePackage{varwidth}
\RequirePackage{enumitem}
\RequirePackage{tensor}
\RequirePackage{mathtools}
\RequirePackage{longtable}
\RequirePackage{multirow}

\def\<{\langle}
\def\>{\rangle}

\newcommand{{\BG}}{\ensuremath{\mathbb {G}}\xspace}

\newcommand{{\BK}}{\ensuremath{\mathbb {K}}\xspace}

\newcommand{\BQ}{\ensuremath{\mathbb {Q}}\xspace}

\newcommand{\BZ}{\ensuremath{\mathbb {Z}}\xspace}

\newcommand{\CL}{\ensuremath{\mathcal {L}}\xspace}

\let\Im\relax
\DeclareMathOperator{\Im}{Im}

\DeclareMathOperator{\Ker}{Ker}

\DeclareMathOperator{\supp}{supp}
\DeclareMathOperator{\Rep}{\mathbf{Rep}}

\newtheorem{theorem}{Theorem}
\newtheorem{proposition}[theorem]{Proposition}
\newtheorem{lemma}[theorem]{Lemma}

\newtheorem{corollary}[theorem]{Corollary}

\theoremstyle{definition}
\newtheorem*{acknowledgement}{Acknowledgement}
\newtheorem{definition}[theorem]{Definition}
\newtheorem{example}[theorem]{Example}

\newtheorem{remark}[theorem]{Remark}

\numberwithin{equation}{section}
\numberwithin{theorem}{section}

\theoremstyle{plain}
\newtheorem{Theorem}{Theorem}

\setitemize[0]{leftmargin=*,itemsep=\the\smallskipamount}
\setenumerate[0]{leftmargin=*,itemsep=\the\smallskipamount}

\renewcommand{\to}{%
   \ifbool{@display}{\longrightarrow}{\rightarrow}%
   }
\let\shortmapsto\mapsto
\renewcommand{\mapsto}{%
   \ifbool{@display}{\longmapsto}{\shortmapsto}%
   }
\newlength{\olen}
\newlength{\ulen}
\newlength{\xlen}
\newcommand{\xra}[2][]{%
   \ifbool{@display}%
      {\settowidth{\olen}{$\overset{#2}{\longrightarrow}$}%
       \settowidth{\ulen}{$\underset{#1}{\longrightarrow}$}%
       \settowidth{\xlen}{$\xrightarrow[#1]{#2}$}%
       \ifdimgreater{\olen}{\xlen}%
          {\underset{#1}{\overset{#2}{\longrightarrow}}}%
          {\ifdimgreater{\ulen}{\xlen}%
             {\underset{#1}{\overset{#2}{\longrightarrow}}}
             {\xrightarrow[#1]{#2}}}}%
      {\xrightarrow[#1]{#2}}
   }
\makeatother
\newcommand{\xyra}[2][]{%
   \settowidth{\xlen}{$\xrightarrow[#1]{#2}$}%
   \ifbool{@display}%
      {\settowidth{\olen}{$\overset{#2}{\longrightarrow}$}%
       \settowidth{\ulen}{$\underset{#1}{\longrightarrow}$}%
       \ifdimgreater{\olen}{\xlen}%
          {\mathrel{\xymatrix@M=.12ex@C=3.2ex{\ar[r]^-{#2}_-{#1} &}}}%
          {\ifdimgreater{\ulen}{\xlen}%
             {\mathrel{\xymatrix@M=.12ex@C=3.2ex{\ar[r]^-{#2}_-{#1} &}}}
             {\mathrel{\xymatrix@M=.12ex@C=\the\xlen{\ar[r]^-{#2}_-{#1} &}}}}}%
      {\mathrel{\xymatrix@M=.12ex@C=\the\xlen{\ar[r]^-{#2}_-{#1} &}}}%
   }
\makeatletter
\newcommand{\xla}[2][]{%
   \ifbool{@display}%
      {\settowidth{\olen}{$\overset{#2}{\longleftarrow}$}%
       \settowidth{\ulen}{$\underset{#1}{\longleftarrow}$}%
       \settowidth{\xlen}{$\xleftarrow[#1]{#2}$}%
       \ifdimgreater{\olen}{\xlen}%
          {\underset{#1}{\overset{#2}{\longleftarrow}}}%
          {\ifdimgreater{\ulen}{\xlen}%
             {\underset{#1}{\overset{#2}{\longleftarrow}}}
             {\xleftarrow[#1]{#2}}}}%
      {\xleftarrow[#1]{#2}}
   }
\newcommand{\isoarrow}{%
   \ifbool{@display}{\overset{\sim}{\longrightarrow}}{\xrightarrow\sim}%
   }
   
\begin{document}

\title[Perfect submonoids of dominant weights]{Perfect submonoids of dominant weights}

\thanks{}
\author[Chengze Duan]{Chengze Duan}
\address[C.Duan]{Department of Mathematics, University of Maryland, College Park, MD, 20740}
\email{cduan12@umd.edu}

\begin{abstract}
Let $G$ be a connected semisimple group. Vinberg introduced the notion of perfect submonoids of dominant weights of $G$ in the study of Vinberg monoids. In this paper, we give explicit descriptions of the perfect submonoids.
\end{abstract}

\maketitle

\section{Introduction}\label{1}

\subsection{Perfect submonoids}



Let $K$ be an algebraic closed field of characteristic 0 and $G$ be a connected reductive group over $K$. Let $T$ be a maximal torus of $G$. Denote the weight lattice and the root lattice of $G$ by $X^*(T)$ and $Q$. Let $X^*_+(T)$ be the set of dominant weights of $G$. For any $\lambda$ in $X^*_+(T)$, we respectively let $L(\lambda)$ be the irreducible representation of $G$ with highest weight $\lambda$. For any two dominant weights $\lambda,\mu$, define
\begin{center}
    $X(\lambda,\mu)=\{ \nu \mid L(\nu)$ is a direct summand of $L(\lambda)\bigotimes L(\mu)\}$.
\end{center}

In the study of the classification of reductive monoids, Vinberg introduced the following definition.

\begin{definition}\label{1.1} \cite[\S 1] {V}
An additive submonoid $L$ of dominant weights is called \textit{perfect} if

\begin{center}
    $\lambda ,\mu \in L$  implies  $X(\lambda ,\mu) \subset L$.
\end{center}
\end{definition}

In this paper, we give a complete characterization of perfect submonoids of dominant weights for connected semisimple groups. We also discuss the perfect submonoids for reductive groups.

\subsection{Main results}
The main result of this paper is the following.

\begin{Theorem}\label{A}
Let $G$ be a connected semisimple algebraic group with a maximal torus $T$.

a) The perfect submonoids of $X^*_+(T)$ with full component support are exactly the intersection of the sublattices of $X^*(T)$ containing $Q$ with $X^*_+(T)$.

b) There is a natural bijection between the perfect submonoids of $X^*_+(T)$ with full component support and the subgroups of the center of $G$.
\end{Theorem}

We refer to Definition \ref{3.6} and Definition \ref{3.13} for the definition of component support. Based on Theorem \ref{A}, one can deduce the characterization for arbitrary perfect submonoids of dominant weights. 

\subsection{Strategy of the proof}
We first reduce the general case to simply connected case by considering the simply connected cover. For any dominant weight $\lambda$ in $L$, there is a special dominant weight $\omega_{\lambda}$ in $L$ by applying $\textit{PRV conjecture}$. Based on the existence of $\omega_{\lambda}$, again by applying $\textit{PRV conjecture}$, we show that if $L$ is a nonzero perfect submonoid of dominant weights, then for any dominant weight $\lambda$ in $L$, the dominant weights which are also weights of $L(\lambda)$ are all contained in $L$. We define the component support for each submonoid of dominant weights. Then we relate the perfect submonoids of dominant weights with full component support to the subgroups of the cocenter and prove Theorem \ref{A} in simply connected case.

Then we prove Theorem \ref{A} based on simply connected case and deduce its corollary for arbitrary perfect submonoids of dominant weights.

At the end of the paper, we look at the connected reductive groups. We also compare our results with the classification of reductive monoids in \cite{V}.

\begin{acknowledgement}
I would like to thank my advisor Xuhua He for suggesting the problem and many helpful discussions. Part of the work was done during my visit to the department of Mathematics and the Institute of mathematical science at the Chinese University of Hong Kong. I also thank Jeffrey Adams, Thomas Haines, Arghya Sadhukhan for discussions.
\end{acknowledgement}

\section{Preliminaries}\label{2}
\subsection{Basic facts about algebraic groups}
Recall that $K$ is algebraically closed of characteristic 0 and $G$ is a connected reductive algebraic group over $K$. Let $T$ be a maximal torus of $G$. The $\textit{root datum}$ of $G$ is a quadruple $(X^*(T), R, X_*(T), R^\vee)$, where $X^*(T)$ is the weight lattice, $X_*(T)$ is the coweight lattice,
$R$ is the set of roots and $R^\vee $ is the corresponding set of coroots.

Let $Q=\mathbb{Z}R$ be the root lattice of $G$. Let $V=X^*(T) \otimes \mathbb{R}$, there is a natural pairing $\langle ,\rangle:X^*(T) \times X_*(T) \to \BZ$ and $P:=\{x \in V \mid \langle x, R^\vee \rangle \subset \mathbb{Z}\}$. Then $Q\subset X^*(T) \subset P$. If $G$ is simply connected, then $X^*(T)=P$.

Choose the set of positive roots $R_+ \subset R$. Let $\Delta = \{ \alpha_1, \alpha_2, ... ,\alpha_n\} \subset R_+$ be the set of simple roots. The fundamental dominant weights with respect to $\Delta$ are $\omega_1,\omega_2,...,\omega_n$. For any two weights $\lambda,\mu$ in $X^*(T)$, write $\mu \preceq \lambda$ if $\mu=\lambda- \sum\limits_{i=1}^{n} k_i \alpha_i$ where $k_i \in \mathbb{Z}_{\geq 0}$ for $1\leq i \leq n$. Let $W$ be the Weyl group of $G$. Then $W$ is generated by simple reflections $\{s_i\}_{i=1}^n$, where $s_i$ acts on $X^*(T)$ by $s_i(\lambda)=\lambda-\langle \lambda, \alpha_i^{\vee} \rangle \alpha_i$, for all $1\leq i \leq n$.

Let $X^*_+(T)$ be the set of dominant weights of $G$. Recall that for any dominant weight $\lambda$ in $X^*_+(T)$, $L(\lambda)$ is the irreducible representation of $G$ with highest weight $\lambda$. Let $L(\lambda)^*$ be its dual representation, which is irreducible with highest weight $\lambda^*$. Denote the set of weights of $L(\lambda)$ by $\Pi(\lambda)$. For any $\mu \in \Pi(\lambda)$, denote the $\mu$-weight space of $L(\lambda)$ by $L(\lambda)_\mu$ and the dimension of $L(\lambda)_\mu$ by $n_\mu(\lambda)$. It is well known that $\mu \in \Pi(\lambda)$ implies $\mu \preceq \lambda$.

By definition,
A subset $\Pi$ of $X^*(T)$ is called $\textit{saturated}$ if for any $\lambda \in \Pi,\alpha \in R$ and $0\leq i\leq \langle \lambda, \alpha^\vee \rangle$, we have $\lambda-i \alpha \in \Pi$. The following properties are well-known, see e.g. \cite[\S 21]{Hum}.

\begin{itemize}
    \item For any $ \lambda' \in \Pi(\lambda)$ and  $w\in W$, we have $w(\lambda') \in \Pi(\lambda)$ and
    $\dim L(\lambda)_{\lambda'}= \dim L(\lambda)_{w(\lambda')}$;
    
    \item $\Pi(\lambda)$ is saturated and if $\mu \in X^*(T)$, then $\mu \in \Pi(\lambda) \text{ is equivalent to that for any } w\in W, w(\mu) \preceq \lambda$. Therefore, $\Pi(\lambda)$ is a finite set and for any dominant weight $ \mu \preceq \lambda$, we have $\mu \in \Pi(\lambda)$.
\end{itemize}

\subsection{Tensor product decomposition}
Let $\lambda,\mu$ be two dominant weights of $G$. We have the tensor product decomposition:

\begin{center}
    $L(\lambda) \bigotimes L(\mu) = \bigoplus\limits_{\nu\in X^*_+(T)} L(\nu)^{\oplus m_{\lambda,\mu}^\nu}$.
\end{center}
Here $m_{\lambda,\mu}^\nu$ is the \textit{tensor product multiplicity}. Recall that $X(\lambda,\mu)$ consists of dominant weights $\nu$, where $L(\nu)$ is a direct summand of $L(\lambda)\bigotimes L(\mu)$. Then $m_{\lambda,\mu}^\nu >0$ if and only if $\nu \in X(\lambda,\mu)$. Therefore, a perfect submonoid of dominant weights is closed under taking direct summands of tensor product.

Recall following classic results describing the possible weights in $X(\lambda,\mu)$.

\begin{lemma} \label{2.1} \cite[Theorem 5.1]{Ko}
Let $\lambda ,\mu,\nu$ be dominant weights in $ X_+^*(T)$. If $\nu \in X(\lambda,\mu)$, then $\nu = \lambda' + \mu$ for some $\lambda' \in \Pi(\lambda)$. In particular, $\nu = \lambda + \mu - \sum\limits_{i=1}^n k_i \alpha_i$, where $k_i \in \mathbb{Z}_{\geq 0}$ for $1\leq i \leq n$. 
\end{lemma}

\begin{lemma} \label{2.2} \cite[\S 24]{Hum}
Let $\lambda ,\mu$ be dominant weights in $X_+^*(T)$. Suppose that for any $\mu' \in \Pi(\mu)$, $\lambda+ \mu'$ is dominant. Then for any $\mu' \in \Pi(\mu)$, $\lambda+ \mu' \in X(\lambda,\mu)$ with multiplicity $m_{\lambda,\mu}^{\lambda+\mu'}= n_{\mu'}(\mu)$.
\end{lemma}

Another key ingredient in our proof is the $\textit{PRV conjecture}$ formulated as a fallout of \cite{PRV}, which was first proved by Kumar. 

\begin{theorem}\cite[Theorem 2.10]{Ku} \label{2.3}
\textit{(PRV conjecture)}
Let $G$ be a semisimple group with Weyl group $W$ over $K$. Let $\lambda,\mu$ be two dominant weights of $G$. For any $w \in W$,  $\overline{\lambda+w\mu} \in X(\lambda,\mu)$, where $\overline{\lambda+w\mu}$ is the only dominant weight in the $W$-orbit of $\lambda+w\mu$. In particular, if $\lambda+w\mu$ is dominant, then $\lambda+w\mu \in X(\lambda,\mu)$.
\end{theorem}

\section{Semisimple case}\label{3}
We prove Theorem \ref{A} in this section.
\subsection{Reduction}\label{3.a}
We first reduce the general case to the case when $G$ is simply connected. Let $G$ be a connected semisimple algebraic group with a maximal torus $T$ and center $Z$. Let $G^{sc}$ be the simply connected cover of $G$ with a maximal torus $T^{sc}$ and center $Z^{sc}$. We know $G\simeq G^{sc}/Z'$, where $Z'$ is a subgroup of $Z^{sc}$. One have
\begin{center}
    $1 \longrightarrow X^*(T) \longrightarrow X^*(T^{sc}) \longrightarrow X^*(Z') \longrightarrow 1$,
\end{center}
and $X^*_+(T)=\{\lambda \in X^*_+(T^{sc}) \mid \lambda|_{Z'}=1\}$ is a subset of $X^*_+(T^{sc})$ by natural inclusion.

Recall that the functor between tensor categories $\Rep(G) \to \Rep(G^{sc})$ is fully faithful by \cite{EGNO}. Then the tensor product multiplicities $m_{\lambda,\mu}^{\nu}$ are the same for $G$ and $G^{sc}$ if $\lambda,\mu,\nu$ are dominant weights of $G$. This can also be seen in \cite[Corollary 3.6]{Ku3}. Therefore, if $L$ is perfect as a submonoid of $X^*_+(T)$, then it is also perfect as a submonoid of $X^*_+(T^{sc})$. Thus we may focus on the case when $G$ is simply connected.

\subsection{Characterization of perfect submonoids of dominant weights}\label{3.b}
Assume $G$ is simply connected in this subsection. Since $G$ is semisimple, there is a decomposition $G=G_1\times \cdot \cdot \cdot \times G_n$, where each $G_k$ is simply connected quasi-simple with a maximal torus $T_k$, center $Z_k$ and Weyl group $W_k$. Let $\Xi=\{1,...,n\}$ be the index set of quasi-simple factors. There are also corresponding decompositions of the weight lattice $X^*(T)=X^*(T_1)\bigoplus \cdot \cdot \cdot \bigoplus X^*(T_n)$ and the root lattice $Q=Q_1\bigoplus \cdot \cdot \cdot \bigoplus Q_n$. We also have
\begin{center}
    $X^*(T)/Q\simeq \bigoplus\limits_{k=1}^n X^*(T_k)/Q_k$.
\end{center}
Let the set of simple roots of $G$ be $\{\alpha_i\}_{i\in I}$ and the corresponding simple reflections be $\{s_i\}_{i\in I}$. Write $I=\bigsqcup\limits_{k=1}^n I_k$, where $I_k$ is the index set of simple roots of $G_k$. 

First we give some perfect submonoids of dominant weights.

\begin{proposition} \label{3.1}
Suppose that $G$ is a simply connected semisimple group. If $\Tilde{L}$ is a sublattice of $X^*(T)$ containing $Q$, then $\Tilde{L}\cap X^*_+(T)$ is a perfect submonoid of $X^*_+(T)$.
\end{proposition}

\begin{proof}
Let $\lambda,\mu$ be two dominant weights in $\Tilde{L}$. For any $\nu \in X(\lambda, \mu)$, by Lemma $\ref{2.1}$, we have $\nu = \lambda + \mu - \sum\limits_{i \in I} k_i \alpha_i$, where $k_i \in \mathbb{Z}_{\geq 0}$ for $i\in I$. Since $\Tilde{L}$ is a lattice containing $Q$, we have $\lambda,\mu$ and $- \sum\limits_{i \in I} k_i \alpha_i$ are all in $\Tilde{L}$. Thus $\nu$ is also in $\Tilde{L}$. Therefore, we have $\nu \in \Tilde{L} \cap X^*_+$ and $L$ is perfect.
\end{proof}

Next we focus on the necessary conditions for perfectness of a submonoid $L\subset X^*_+(T)$. By above decomposition of weight lattice, any weight $\lambda \in X^*(T)$ can be denoted by $\big(\pi_1(\lambda),...,\pi_n(\lambda)\big)$, where $\pi_k:X^*(T)\to X^*(T_k)$ is the canonical projection, for $1\leq k\leq n$. Suppose that $\lambda$ is dominant, define the \textit{support} of $\lambda$ as
\begin{center}
        $\supp(\lambda)=\{i \in I\mid \langle \lambda, \alpha_i^{\vee}  \rangle > 0\}$.
\end{center}
For any $1\leq k\leq n$, say $\lambda$ is $k$-\textit{regular} if $\supp(\lambda)\supset I_k$. If $\lambda$ is $k$-\textit{regular} for all $k$, $1\leq k\leq n$, then $\lambda$ is a $\textit{regular}$ dominant weight in $X^*(T)$.

\begin{definition}\label{3.2}
Let $G$ be a semisimple group. For any dominant weight $\lambda$ of $G$, the \textit{component support} of $\lambda$ is the set $\{1\leq k\leq n \mid \pi_k(\lambda) \text{ is nontrivial}\}$.
\end{definition} 

Let $L$ be a perfect submonoid of $X^*_+(T)$. It is clear that for any $1\leq k\leq n$, $\pi_k(L)$ is a perfect submonoid of $X^*_+(T_k)$. We claim the existence of some certain $k$-regular dominant weights in a nonzero perfect submonoid $L$ of dominant weights. 

\begin{lemma} \label{3.3}
Suppose that $G$ is simply connected semisimple and $L$ is a nonzero perfect submonoid of $X^*_+(T)$. Let $\lambda$ be a dominant weight in $L$. Then there exists a dominant weight $\omega_{\lambda} \in L$ such that for any $1\leq k\leq n$, $\omega_{\lambda}$ is $k$-regular if $\pi_k(\lambda)$ is nontrivial.
\end{lemma}
\begin{proof}
It suffices to prove the lemma for quasi-simple group $G$. Indeed, suppose that for any $1\leq k\leq n$ such that $\pi_k(\lambda)$ is nontrivial, there is a $k$-regular dominant weight $\pi_k(\mu_k)\in \pi_k(L)$, where $\mu_k\in L$. Then $\sum\limits_{k, \pi_k(\lambda) \text{ is nontrivial }} \mu_k$ is a desirable dominant weight $\omega_{\lambda}\in L$.

Assume that $G$ is quasi-simple. It suffices to show that for any dominant weight $\lambda\in L$ with $\supp(\lambda)\subsetneq I$, there is another dominant weight $\mu\in L$ such that $\supp(\mu)\supsetneq \supp(\lambda)$.

Let $\mathcal{D}$ be the Dynkin diagram of $G$ and $\lambda$ be a dominant weight in $L$ with $\supp(\lambda)\subsetneq I$. There are vertices $j\in \supp(\lambda)$ and $i_1 \notin \supp(\lambda)$ such that $j$ and $i_1$ are joint with each other in $\mathcal{D}$. Then $\langle \lambda,\alpha_j^{\vee} \rangle>0$. Let $I'=\{i_1,i_2,...,i_m,j\}$ be the subset of $I$ consisting of all vertices joint with $j$ and $j$ itself. Consider the weight $\mu=2\lambda+s_j(\lambda)$. We show it is dominant.

For any $i\in I$, we have
\begin{center}
        $\langle \mu, \alpha_i^{\vee} \rangle=3\langle \lambda,\alpha_i^{\vee}\rangle -\langle \lambda,\alpha_j^{\vee}\rangle \langle \alpha_j,\alpha_i^{\vee} \rangle$.
\end{center}
If $i=j$, then $\langle \mu, \alpha_j^{\vee} \rangle=\langle \lambda, \alpha_j^{\vee} \rangle>0$. If $i\in I'\setminus \{j\}$, then $\langle \mu, \alpha_i^{\vee} \rangle> 3\langle \lambda, \alpha_i^{\vee} \rangle\geq 0$ since $\langle \alpha_j,\alpha_i^{\vee} \rangle<0$. If $i\in I\setminus I'$, then $\langle \mu, \alpha_i^{\vee} \rangle=3\langle \lambda,\alpha_i^{\vee}\rangle\geq 0$ since $\langle \alpha_j,\alpha_i^{\vee} \rangle=0$. By above computations, we have $\mu$ is dominant. Then by Theorem \ref{2.3}, we have $\mu\in X(2\lambda,\lambda)$ is contained in $L$.

Now we look at the support. Still by above computations, for $i\in I\setminus I'$, we have $i\in \supp(\mu)$ if and only if $i\in \supp(\lambda)$. We also have $\supp(\mu)$ contains $I'$ while $i_1\notin \supp(\lambda)$. Therefore, we have $\supp(\mu)\supsetneq \supp(\lambda)$ and the lemma is proved.
\end{proof}

Based on above property of $\omega_{\lambda}$ and the fact that $\Pi(\lambda)$ is a finite set, we have a direct corollary.

\begin{corollary} \label{3.4}
Suppose that $G$ is simply connected semisimple and $L$ is a nonzero perfect submonoid of $X^*_+(T)$. Let $\lambda$ be a dominant weight in $L$. Then there is a positive integer $m$ such that $\mu + m \omega_{\lambda} \in L$ for any weight $\mu \in \Pi(\lambda)$.
\end{corollary}

We also need the following technical proposition, which will be proved in Section \ref{4}.

\begin{proposition} \label{3.5}
Suppose that $G$ is simply connected semisimple. If $L$ is a nonzero perfect submonoid of $X_+^*(T)$, then for any $\lambda \in L$, all the dominant weights in $\Pi(\lambda)$ are contained in $L$.
\end{proposition}

For the proof of Proposition \ref{3.5} and our later discussions, we cannot reduce them directly to the case when $G$ is quasi-simple. This is because $\big(\pi_1(\lambda_1),\pi_2(\lambda_2),...,\pi_n(\lambda_n)\big)$ may not be in $L$ even if $\lambda_1,...,\lambda_n$ are all in $L$.

\begin{definition}\label{3.6}
Let $L$ be a submonoid of $X^*_+(T)$, the \textit{component support} of $L$ is the set $\{1\leq k\leq n \mid \pi_k(L)\neq \{0\}\}$.
If the component support of $L$ is equal to $\{1,2,...,n\}$, then $L$ is said to have \textit{full component support}. In particular, when $G$ is quasi-simple, every nonzero submonoid of $X^*_+(T)$ has full component support.
\end{definition}

We first restrict ourselves to perfect submonoids of $X^*_+(T)$ with full component support.

\begin{lemma}\label{3.7}
Suppose that $G$ is simply connected semisimple. If $L$ is a perfect submonoid of $X^*_+(T)$ with full component support, then for any $1\leq k\leq n$, we have $Q_{k}\cap X^*_+(T)$ is contained in $L$. In particular, $Q\cap X^*_+(T)$ is contained in $L$.
\end{lemma}

\begin{proof}
Let $\mu$ be arbitrary in $Q_k\cap X^*_+(T)$. Since $L$ has full component support, there is a dominant weight $\lambda$ in $L$ with full component support. By Lemma \ref{3.3}, there is a regular dominant weight $\omega_{\lambda}=\big(\pi_{1}(\omega_{\lambda}),...,\pi_{n}(\omega_{\lambda})\big)$ in $L$. We know that $\omega_{\lambda}$ is a $\BQ_{\geq 0}$-combination of simple roots. Then one can take a positive integer $m$ such that $m\omega_{\lambda} \in Q\cap X^*_+(T)$. Moreover, since $\omega_{\lambda}$ is regular, we have $\langle \omega_{\lambda}, \alpha_i^{\vee} \rangle >0$ for any $i\in I_k$. One can take $m$ large enough such that $\langle m\omega_{\lambda}, \alpha_i^{\vee} \rangle\geq \langle \mu, \alpha_i^{\vee} \rangle$ for any $i\in I_k$. Then $m\pi_k(\omega_{\lambda})-\mu$ is a $\BZ_{\geq 0}$-combination of simple roots in $Q_k$. Moreover, we have $m\omega_{\lambda}-\mu:=\big(m\pi_1(\omega_\lambda),...,m\pi_{k-1}(\omega_\lambda),m\pi_k(\omega_{\lambda})-\mu,m\pi_{k+1}(\omega_\lambda),...,m\pi_{n}(\omega_\lambda)\big)$ is a $\BZ_{\geq 0}$-combination of simple roots and thus $\mu\preceq m\omega_{\lambda}$. Then $\mu \in \Pi(m\omega_{\lambda})\cap X^*_+(T)$ is in $L$ by Proposition \ref{3.5}. Therefore, $Q_k\cap X^*_+(T)$ is contained in $L$.

In particular, since $L$ is a submonoid of $X^*_+(T)$, we have $Q\cap X^*_+(T)$ is contained in $L$ by adding $Q_k\cap X^*_+(T)$ for $1\leq k\leq n$.
\end{proof}

Based on the above lemma, we use the cocenter to characterize the perfect submonoids of $X^*_+(T)$. Consider the canonical projection map $p:X^*(T)\to X^*(T)/Q$. If $L$ is a perfect submonoid of $X^*_+(T)$ with full component support, then $p(L)$ is a subgroup of $X^*(T)/Q$ since $X^*(T)/Q$ is finite.

\begin{proposition}\label{3.8}
Suppose that $G$ is simply connected semisimple and $L$ is a perfect submonoid of $X^*_+(T)$ with full component support. Then $L=p^{-1}(\CL)\cap X^*_+(T)$ for some subgroup $\CL$ of $X^*(T)/Q$.

\end{proposition}
\begin{proof}
Let $\CL=p(L)$ be a subgroup of $X^*(T)/Q$. By definition we have $L\subset p^{-1}(\CL)\cap X^*_+(T)$. Then it suffices to show:
\begin{center}
    For any $a \in \CL$, $p^{-1}(a)\cap X^*_+(T)$ is contained in $L$.
\end{center}
Indeed, by our choice of $\CL$, there exists $\lambda \in L$ such that $p(\lambda)=a$. Let $\mu$ be an arbitrary dominant weight in $p^{-1}(a)$. Then $\lambda-\mu\in Q$. By same argument as in the proof of Lemma \ref{3.7}, there exists a regular dominant weight $\omega\in L$. Then there is a positive integer $m$ such that $\lambda-\mu+m\omega$ is dominant by regularity of $\omega$. Then $\lambda-\mu+m\omega \in Q\cap X^*_+(T)$ is a $\BZ_{\geq 0}$-combination of simple roots. By Proposition \ref{3.5}, $\mu \in \Pi(\lambda+m\omega)\cap X^*_+(T)$ is in $L$. Therefore, we have $p^{-1}(a)\cap X^*_+(T)\subset L$ and $L=p^{-1}(\CL)\cap X^*_+(T)$.
\end{proof}

Based on above, we can give the characterization of perfect submonoids of dominant weights. 

\begin{proposition}\label{3.9}
Let $G$ be a simply connected semisimple group. The perfect submonoids of $X^*_+(T)$ with full component support are exactly $\Tilde{L}\cap X^*_+(T)$, where $\Tilde{L}$ is any sublattice of $X^*(T)$ containing $Q$.
\end{proposition}
\begin{proof}
By Proposition \ref{3.1}, the intersection of sublattices of $X^*(T)$ containing $Q$ with $X^*_+(T)$ are perfect. Moreover, these perfect submonoids clearly have full component support since $Q\cap X^*_+(T)$ has full component support.

Let $L$ be a perfect submonoid of $X^*_+(T)$ with full component support. By Proposition \ref{3.8}, $L=p^{-1}(\CL)\cap X^*_+(T)$ for some subgroup $\CL$ of $X^*(T)/Q$. We also have $p^{-1}(\CL)$ is a subgroup of $X^*(T)$. Moreover, $p^{-1}(\CL)$ contains $p^{-1}(0)=Q$. Therefore, the perfect submonoid $L$ is the intersection of a sublattice $p^{-1}(\CL)$ of $X^*(T)$ containing $Q$ with $X^*_+(T)$.
\end{proof}

\subsection{Reformulation of the characterization} \label{3.c}
In Proposition $\ref{3.8}$, we relates our perfect submonoids of $X^*_+(T)$ with the cocenter of $G$. Now we give a reformulation of perfect submonoids of dominant weights using central characters. Still assume $G$ is simply connected in this subsection. Keep the notations in Subsection \ref{3.b}. 

Let $L$ be an arbitrary perfect submonoid of $X^*_+(T)$ with full component support. Define a subset $Z_L$ of $Z$ as
\begin{center}
    $Z_L=\{z \in Z \mid \lambda(z)=1 ,\forall \lambda \in L\}$.
\end{center}
Since $Z_L=\bigcap\limits_{\lambda \in L} \Ker(\lambda|_{Z})$, we have that $Z_L$ is a subgroup of $Z$.

Conversely, let $Z'$ be an arbitrary subgroup of $Z$. Define a subset $L_{Z'}$ of $X^*_+(T)$ as
\begin{center}
    $L_{Z'}=\{\lambda\in X^*_+(T) \mid \lambda|_{Z'}=1\}$.
\end{center}
Then $L_{Z'}$ is a perfect submonoid of $X^*_+(T)$ with full component support. Indeed, there is a unique (up to isomorphism) connected algebraic group $G'$ with simply connected cover $G$ such that $G'\simeq G/Z'$. By \cite[\S 1.2]{BT}, the maximal torus $T'$ of $G'$ satifying
\begin{center}
    $1 \longrightarrow Z' \longrightarrow T \longrightarrow T' \longrightarrow 1$
\end{center}
gives rise to
\begin{center}
    $1 \longrightarrow X^*(T') \longrightarrow X^*(T) \longrightarrow X^*(Z') \longrightarrow 1$,
\end{center}
and $X^*_+(T')= \{\lambda\in X^*_+(T) \mid \lambda|_{Z'}=1 \}=L_{Z'}$. Moreover, as weights in $Q\cap X^*_+(T)$ act trivially on $Z\supset Z'$, we have $L_{Z'}\supset Q\cap X^*_+(T)$ and $L_{Z'}$ is a perfect submonoid of $X^*_+(T)$ with full component support.

\begin{proposition}\label{3.10}
Let $G$ be a simply connected semisimple group. The maps $\varphi: L\mapsto Z_L$, $\psi:Z'\mapsto L_{Z'}$ give a natural bijection between the perfect submonoids of $X^*_+(T)$ with full component support and the subgroups of $Z$.
\end{proposition}

\begin{proof}
Let $L$ be an arbitrary perfect submonoid of $X^*_+(T)$ with full component support. By Proposition \ref{3.9}
, we have $L=\Tilde{L}\cap X^*_+(T)$ for some sublattice $\Tilde{L}$ of $X^*(T)$ containing $Q$. Then there is a unique (up to isomorphism) connected semisimple group $G'$ with simply connected cover $G$ and a maximal torus $T'$ such that $X^*(T')=\Tilde{L}$. Since $G'\simeq G/Z'$ for a unique subgroup $Z'$ of $Z$, we have $X^*(T')=\{\lambda \in X^*(T) \mid \lambda|_{Z'}=1\}$. Then we have $L=\Tilde{L}\cap X^*_+(T)=\{\lambda \in X^*_+(T) \mid \lambda|_{Z'}=1\}=\psi(Z')$ and $\psi$ is surjective. Meanwhile, by uniqueness of $G'$, $\psi$ is injective.

Now we show that $\varphi$ and $\psi$ are inverse to each other. Consider $(\psi \circ \varphi)(L)$ is also a perfect submonoid of $X^*_+(T)$ with full component support. For any $\lambda \in L$ and any $z \in \varphi(L)$, we have $\lambda(z)=1$. Then by definition, $\lambda$ is in $ (\psi \circ \varphi)(L)$ and $L\subset(\psi \circ \varphi)(L)$. Meanwhile, since $\psi$ is surjective, $L=\psi (Z')$ for some subgroup $Z'$ of $Z$. Then $\varphi(L)$ contains $Z'$. Then $(\psi \circ \varphi)(L)$ is a subset of $ \psi(Z')=L$. Therefore, we have $(\psi \circ \varphi)(L)=L$. For any $Z'<Z$, we have $(\psi \circ \varphi \circ \psi)(Z')=\psi(Z')$ by above. Since $\psi$ is injective, we have $(\varphi \circ \psi)(Z')=Z'$. Therefore, the pair $(\varphi,\psi)$ gives a bijection and it is clearly natural by definition.
\end{proof}

\subsection{Proof of the main result}\label{3.d}
Now we return to the setting in Subsection \ref{3.a} and prove Theorem \ref{A}. Let $L$ be a submonoid of $X^*_+(T)$. We first define the component support of $L$.
\begin{definition}\label{3.13}
Let $L$ be a submonoid of $X^*_+(T)$. The \textit{component support} of $L$ is the component support of $L$ as a submonoid of $X^*_+(T^{sc})$ (see Definition \ref{3.6}).
\end{definition}
\textbf{a)}
Let $L$ be a perfect submonoid of $X^*_+(T)$ with full component support. By our discussion above, $L$ is also a perfect submonoid of $X^*_+(T^{sc})$ with full component support. Therefore, by Proposition $\ref{3.9}$, we have $L=\Tilde{L}\cap X^*_+(T^{sc})$ where $\Tilde{L}$ is a sublattice of $X^*(T^{sc})$ containing the root lattice $Q$. Since $L$ is contained in $X^*_+(T)$, we have $L=\Tilde{L}\cap X^*_+(T)$. One can also write $L=\big(\Tilde{L}\cap X^*(T)\big)\cap X^*_+(T)$. Clearly, $\Tilde{L}\cap X^*(T)$ is a sublattice of $X^*(T)$ containing $Q$.

Conversely, let $L=\Tilde{L}\cap X^*_+(T)$ where $\Tilde{L}$ is a sublattice of $X^*(T)$ containing $Q$. Then $\Tilde{L}$ is also a sublattice of $X^*(T^{sc})$ containing $Q$. We also know $\Tilde{L}\cap X^*_+(T)=\Tilde{L}\cap X^*_+(T^{sc})$ since $\Tilde{L}\subset X^*(T)$. Then by Proposition \ref{3.9}, $L$ is a perfect submonoid of $X^*_+(T^{sc})$ with full component support and is also a perfect submonoid of $X^*_+(T)$ with full component support.

\textbf{b)}
We know that $G$ is isomorphic to $G^{sc}/Z'$. We also have $Z\simeq Z^{sc}/Z'$ which is natural. Then it suffices to show: There is a natural bijection between the perfect submonoids of $X^*_+(T)$ with full component support and the subgroups of $Z^{sc}/Z'$.

Recall that perfect submonoids of $X^*_+(T)$ with full component support are also perfect submonoids of $X^*_+(T^{sc})$ with full component support. By Proposition \ref{3.10}, there is a natural bijection between perfect submonoids of $X^*_+(T^{sc})$ with full component support and subgroups of $Z^{sc}$ given by
\begin{center}
    $\varphi: L\longmapsto Z^{sc}_{L}=\{z\in Z^{sc} \mid \lambda(z)=1, \forall \lambda \in L\}$,
\end{center}
and its inverse
\begin{center}
    $\psi: (Z^{sc})' \longmapsto L_{(Z^{sc})'}=\{\lambda \in X^*_+(T^{sc})\mid \lambda|_{(Z^{sc})'}=1\}$.
\end{center}

Note that 
\begin{center}
       $X^*_+(T)=\{\lambda \in X^*_+(T^{sc})\mid \lambda|_{Z'}=1 \}$.
\end{center}
If $L$ is a perfect submonoid of $X^*_+(T)$ with full component support, then $\varphi(L)=Z^{sc}_L$ contains $Z'$. 

Conversely, for any subgroup $(Z^{sc})'$ of $Z^{sc}$ containing $Z'$, $\psi((Z^{sc})')=L_{(Z^{sc})'}$ is actually a perfect submonoid of $X^*_+(T)$. Then the restrictions of $\varphi$ and $\psi$ actually give a natural bijection between perfect submonoids of $X^*_+(T)$ with full component support and subgroups of $Z^{sc}$ containing $Z'$. Since there is a natural bijection between subgroups of $Z^{sc}$ containing $Z'$ and subgroups of $Z^{sc}/Z'$, one can combine two natural bijections together and get the required bijection.

\subsection{Characterization for arbitrary perfect submonoids}\label{3.e}
In this subsection we drop the assumption that $L$ has full component support and deal with arbitrary perfect submonoids. Indeed, we only need to consider the nonzero perfect submonoids.

Let $\Xi_0$ be an arbitrary nonempty subset of $\Xi$ and $L$ be a perfect submonoid of $X^*_+(T)$ with component support $\Xi_0$. Then $L$ is also a perfect submonoid of $X^*_+(T^{sc})$ with component support $\Xi_0$. Let $X^*(T)_{\Xi_0}=X^*(T)\cap X^*(T^{sc})_{\Xi_0}$ and $Q_{\Xi_0}=Q\cap X^*(T^{sc})_{\Xi_0}$, where
\begin{center}
    $X^*(T^{sc})_{\Xi_0}:=\{\big(\pi_1(\lambda),...,\pi_n(\lambda)\big)\mid \lambda\in X^*(T^{sc}), \pi_k(\lambda)=0 \text{ for any } k\notin \Xi_0\}\subset X^*(T^{sc})$.
\end{center}
It is clear that $X^*(T^{sc})_{\Xi_0}$ and $Q_{\Xi_0}$ are isomorphic to the weight lattice and the root lattice of $G^{sc}_{\Xi_0}=\prod\limits_{k\in \Xi_0} G^{sc}_k$, respectively. Then $L$ is contained in $X^*_+(T)_{\Xi_0}\subset X^*_+(T^{sc})_{\Xi_0}$ and one can view $L$ as a perfect submonoid of dominant weights of $G^{sc}_{\Xi_0}$ with full component support. 

Then we can slightly modify the maps $\varphi$ and $\psi$. Recall that $G\simeq G^{sc}/Z'$. Let $Z^{sc}_{\Xi_0}=\prod\limits_{k\in \Xi_0} Z^{sc}_k$ and $Z_{\Xi_0}'=Z'\cap Z^{sc}_{\Xi_0}$. Define the map $\varphi_{\Xi_0}$ from perfect submonoids of $X^*_+(T^{sc})$ with component support $\Xi_0$ to subgroups of $Z^{sc}_{\Xi_0}$ as $\varphi_{\Xi_0}(L)=(Z^{sc}_{\Xi_0})_L$, where
\begin{center}
     $(Z^{sc}_{\Xi_0})_L=\{z \in Z^{sc}_{\Xi_0} \mid \lambda(z)=1 ,\forall \lambda \in L\}$.
\end{center}
For the inverse direction, define the map $\psi_{\Xi_0}$ as $\psi_{\Xi_0}\big((Z^{sc}_{\Xi_0})'\big)=L_{(Z^{sc}_{\Xi_0})'}$, where
\begin{center}
    $L_{(Z^{sc}_{\Xi_0})'}=\{\lambda \in X^*_+(T^{sc})_{\Xi_0} \mid \lambda|_{(Z^{sc}_{\Xi_0})'}=1\}$.    
\end{center}
One can also write
\begin{center}
    $L_{(Z^{sc}_{\Xi_0})'}=\{\lambda \in X^*_+(T^{sc}) \mid \lambda|_{(Z^{sc}_{\Xi_0})'}=1, \lambda|_{T_k}=1,\forall k\notin \Xi_0 \}$.
\end{center}
Then we deduce the characterization for perfect submonoids of $X^*_+(T)$ with component support $\Xi_0\subset \Xi$ and its reformulation as a corollary of Theorem \ref{A}. 
\begin{corollary}\label{3.12}
Let $G$ be a connected semisimple algebraic group. Then

a) The perfect submonoids of $X^*_+(T)$ with component support $\Xi_0 \subset \Xi$ are exactly $\Tilde{L}\cap X^*_+(T)$, where $\Tilde{L}$ is any sublattice of $X^*(T)_{\Xi_0}$ containing $Q_{\Xi_0}$;

b) There is a natural bijection between the perfect submonoids of $X^*_+(T)$ with component support $\Xi_0 \subset \Xi$ and the subgroups of $Z^{sc}_{\Xi_0}/Z_{\Xi_0}'$.
\end{corollary}
\begin{proof}
a) One notices that $X^*(T)_{\Xi_0}$ is a sublattice of $X^*(T^{sc})_{\Xi_0}$ containing $Q_{\Xi_0}$. Then $X^*(T)_{\Xi_0}$ and $Q_{\Xi_0}$ are the weight lattice and the root lattice of a connected semisimple group $G_{\Xi_0}'$ with simply connected cover $G^{sc}_{\Xi_0}$, respectively. Then one can check perfect submonoids of $X^*_+(T)_{\Xi_0}$ with full component support are also perfect submonoids of $X^*_+(T^{sc})$ contained in $X^*_+(T)$ with component support $\Xi_0$. Then by our discussions above, perfect submonoids of $X^*_+(T)$ with component support $\Xi_0$ are exactly perfect submonoids of $X^*_+(T)_{\Xi_0}$ with full component support.

We know $X^*_+(T)_{\Xi_0}$ is the set of dominant weights of $G_{\Xi_0}'$. Then by applying Theorem \ref{A} to perfect submonoids of $X^*_+(T)_{\Xi_0}$ with full component support, part a) is proved. 

b) Identify $X^*_+(T^{sc})_{\Xi_0}$ with the set of dominant weights of $G^{sc}_{\Xi_0}$. Then there is a natural bijection between the perfect submonoids of $X^*_+(T^{sc})$ with component support $\Xi_0$ and the perfect submonoids of $X^*_+(T^{sc})_{\Xi_0}$ with full component support. Then by applying Proposition \ref{3.10} to $G^{sc}_{\Xi_0}$, we have the maps $\varphi_{\Xi_0}$ and $\psi_{\Xi_0}$ give a natural bijection between the perfect submonoids of $X^*_+(T^{sc})$ with component support $\Xi_0$ and the subgroups of $Z^{sc}_{\Xi_0}$. 

Moreover, same as the proof of Theorem \ref{A}, the restrictions of $\varphi_{\Xi_0}$ and $\psi_{\Xi_0}$ actually  give a natural bijection between perfect submonoids of $X^*_+(T)$ with component support $\Xi_0$ and subgroups of $Z_{\Xi_0}^{sc}$ containing $Z_{\Xi_0}'$. Since there is a natural bijection between subgroups of $Z_{\Xi_0}^{sc}$ containing $Z_{\Xi_0}'$ and subgroups of $Z_{\Xi_0}^{sc}/Z_{\Xi_0}'$, again we can combine two bijections together and get the required natural bijection.
\end{proof}

\section{Proof of Proposition \ref{3.5}}\label{4}
In this section, we keep the notations in Subsection \ref{3.b} and prove Proposition \ref{3.5}. We first give the idea of the proof. Then we reduce it to the case when $G$ is quasi-simple and finally give the computations in different types.
\subsection{Idea}\label{4.a}
Let $\lambda$ be a dominant weight in $L$ with component support $\Xi_0\subset \Xi$. We may assume $\Xi_0=\{1,2,...,n_0\}$. Take the dominant weight $ \omega_{\lambda} \in L$ in Lemma \ref{3.3} such that for any $1\leq k\leq n_0$, $\omega_{\lambda}$ is $k$-regular. By Lemma \ref{3.4}, there is a positive integer $m$ such that for any $\mu \in \Pi(\lambda)$, $\mu+m\omega_{\lambda}$ is in $L$. Without loss of generality, we assume $m=1$. 

Now let $\mu$ be an arbitrary dominant weight in $\Pi(\lambda)$. Then the component support of $\mu$ is contained in $\Xi_0$. Our idea is finding a dominant weight $\eta$ in $L$ based on $\omega_{\lambda}$, such that $\mu+\eta$ is also in $L$ and $w_0(\eta)=-\eta$, where $w_0$ is the longest element in $W$. Then we have $\mu=\mu+\eta+w_0(\eta)\in X(\mu+\eta,\eta)\subset L$ by Theorem \ref{2.3}. Since $\mu$ is arbitrary, all dominant weights in $\Pi(\lambda)$ are contained in $L$, which proves Proposition \ref{3.5}.

\subsection{Reduction}\label{4.b}
As in Subsection \ref{3.b}, one can write $\omega_{\lambda}=\big(\pi_1(\omega_{\lambda}),...,\pi_n(\omega_{\lambda})\big)$, where $\pi_k(\omega_{\lambda})\in X^*_+(T_k)$. For any $k\notin \Xi_0$, we know $\pi_k(\omega_{\lambda})=0$. For any $w\in W$, we write $w=(w^{(1)},w^{(2)},...,w^{(n)})$, where $w^{(k)}$ is in the Weyl group $W_k$ of $G_k$. In particular, $w_0=(w_0^{(1)},w_0^{(2)},...,w_0^{(n)})$, where $w_0^{(k)}$ is the longest element in $W_k$. We construct $\eta$ by some lemmas.

\begin{lemma}\label{a.1}
Suppose that $G$ is simply connected quasi-simple and $\omega \in L$ is regular. There is a sequence $\{\nu_0=\omega,\nu_1,...,\nu_r\}$ of nonzero dominant weights in $L$ such that $w_0(\nu_r)=-\nu_r$ and for any $0\leq l \leq r-1$, $\nu_{l+1}=\beta_l+\sigma_l(\gamma_l)$ for some $\beta_l,\gamma_l \in \{\nu_0,...,\nu_l\}$ and $\sigma_l\in W$.
\end{lemma}

\begin{proof}
We give the precise computations for this lemma in different types in Subsection \ref{4.c}.
\end{proof}

\begin{lemma}\label{a.2}
Let $1 \leq k_0 \leq n_0$ and $\omega=\big(\pi_1(\omega),...,\pi_n(\omega)\big)$ be any dominant weight in $L$ with component support $\Xi_0$ such that, $\mu+\omega\in L$ and $\omega$ is $k_0$-regular. Then there is a dominant weight $\theta_{k_0}=\big(\pi_1(\theta_{k_0}),...,\pi_n(\theta_{k_0})\big)$ in $L$ with component support $\Xi_0$ such that $\mu+\theta_{k_0} \in L$, $w_0^{(k_0)}\big(\pi_{k_0}(\theta_{k_0})\big)=-\pi_{k_0}(\theta_{k_0})$ and $\pi_k(\theta_{k_0})$ is a positive integral multiple of $\pi_k(\omega)$ for $k \in \Xi_0 \setminus \{k_0\}$.
\end{lemma}
\begin{proof}
Recall that $\pi_{k_0}(L)$ is a perfect submonoid of $X^*_+(T_{k_0})$ with full component support since $L$ is perfect and $G_{k_0}$ is quasi-simple. Then there is a regular dominant weight $\pi_{k_0}(\omega)$ of $G_{k_0}$ in $\pi_{k_0}(L)$. Take the sequence $\{\nu_0=\pi_{k_0}(\omega),...,\nu_r\}$ in $\pi_{k_0}(L)$ in Lemma \ref{a.1}. For any $0\leq l \leq r-1$, one can write $\nu_{l+1}=\beta_l+\sigma_l^{(k_0)}(\gamma_l)$ for some $\beta_l,\gamma_l \in \{\nu_0,...,\nu_l\}$ and $\sigma_l^{(k_0)}\in W_{k_0}$.

Based on above sequence, there is also a sequence $\{\Tilde{\nu}_0,\Tilde{\nu}_1,...,\Tilde{\nu}_r\}$ of nonzero weights in $X^*(T)$ as following: 

1) $\Tilde{\nu}_0=\omega$;

2) Suppose that we have $\{\Tilde{\nu}_0,\Tilde{\nu}_1,...,\Tilde{\nu}_l\}$ and $\nu_{l+1}=\nu_{a_{l+1}}+\sigma_l^{(k_0)}(\nu_{b_{l+1}})$ for some $a_{l+1},b_{l+1} \in \{0,1,...,l\}$. Then set 
\begin{center}
    $\Tilde{\nu}_{l+1}:=\Tilde{\nu}_{a_{l+1}}+\Tilde{\sigma_l}(\Tilde{\nu}_{b_{l+1}})$,
\end{center}
where $\Tilde{\sigma}_l=(\Tilde{\sigma}_l^{(1)},...,\Tilde{\sigma}_l^{(n)})$ is given by $\Tilde{\sigma}_l^{(k)}=\left\{ \begin{array}{cc}
   \sigma_l^{(k_0)},&k=k_0 \\
   Id^{(k)},&k\neq k_0
\end{array}\right.$.

Then we claim that for any $0\leq l\leq r$, both $\Tilde{\nu}_l$ and $\mu+\Tilde{\nu}_l$ are in $L$ with component support $\Xi_0$. Moreover, $\pi_{k_0}(\Tilde{\nu}_l)=\nu_l$ and $\pi_k(\Tilde{\nu}_l)$ is a positive integral multiple of $\pi_k(\omega)$ for $k \in \Xi_0\setminus \{k_0\}$.

Prove the claim by induction on $l$. For $l=0$, the claim is clearly true since $\Tilde{\nu}_0=\omega$. Suppose that the claim is true for $\{0,1,...,l\}$. Since $\Tilde{\nu}_{l+1}=\Tilde{\nu}_{a_{l+1}}+\Tilde{\sigma_l}(\Tilde{\nu}_{b_{l+1}})$, we have
\begin{center}
    $\pi_k(\Tilde{\nu}_{l+1})=\left\{ \begin{array}{ll}
   \pi_{k_0}(\Tilde{\nu}_{a_{l+1}})+\Tilde{\sigma_l}^{(k_0)}\big(\pi_{k_0}(\Tilde{\nu}_{b_{l+1}})\big)=\nu_{a_{l+1}}+\sigma_l^{(k_0)}(\nu_{b_{l+1}})=\nu_{l+1},& k=k_0, \\
   \pi_k(\Tilde{\nu}_{a_{l+1}})+\Tilde{\sigma_l}^{(k)}\big(\pi_k(\Tilde{\nu}_{b_{l+1}})\big)=\pi_k(\Tilde{\nu}_{a_{l+1}})+\pi_k(\Tilde{\nu}_{b_{l+1}}),& k\in \Xi_0\setminus k_0, \\
   \pi_k(\Tilde{\nu}_{a_{l+1}})+\Tilde{\sigma_l}^{(k)}\big(\pi_k(\Tilde{\nu}_{b_{l+1}})\big)=0,& k\notin \Xi_0.
\end{array}\right.$
\end{center}
We have $\pi_k(\Tilde{\nu}_{l+1})=\pi_k(\Tilde{\nu}_{a_{l+1}})+\pi_k(\Tilde{\nu}_{b_{l+1}})$ is a positive integral multiple of $\pi_k(\omega)$ for $k\in \Xi_0\setminus \{k_0\}$ by induction hypothesis. Since $\pi_{k_0}(\Tilde{\nu}_{l+1})=\nu_{l+1}$ is dominant, we have $\Tilde{\nu}_{l+1}$ and $\mu+\Tilde{\nu}_{l+1}$ are both dominant with component support $\Xi_0$ by above computation. Then we have $\Tilde{\nu}_{l+1} \in X(\Tilde{\nu}_{a_{l+1}},\Tilde{\nu}_{b_{l+1}})$ and $\mu+\Tilde{\nu}_{l+1} \in X(\mu+\Tilde{\nu}_{a_{l+1}},\Tilde{\nu}_{b_{l+1}})$ by Theorem \ref{2.3}. Therefore, $\Tilde{\nu}_{l+1}$ and $\mu+\Tilde{\nu}_{l+1}$ are both in $L$. Then the claim is true for $(l+1)$-case. By induction, the claim is true.

Consider the dominant weight $\Tilde{\nu}_r$ in the sequence. From above claim, We know $\Tilde{\nu}_r$ and $\mu+\Tilde{\nu}_r$ are both in $L$ with component support $\Xi_0$ and for $k\in \Xi_0\setminus \{k_0\}$, $\pi_k(\Tilde{\nu}_r)$ is a positive integral multiple of $\pi_k(\omega)$. Since $\pi_{k_0}(\Tilde{\nu}_r)=\nu_r$, we also have $w_0^{(k_0)}\big(\pi_{k_0}(\Tilde{\nu}_r)\big)=-\pi_{k_0}(\Tilde{\nu}_r)$ by Lemma \ref{a.1}. Then $\Tilde{\nu}_r$ is a desirable dominant weight $\theta_{k_0}$.
\end{proof}

\begin{lemma}\label{a.3}
There is a sequence $\{\eta_0=\omega_{\lambda},\eta_1, \eta_2,...,\eta_{n_0}\}$ of dominant weights with component support $\Xi_0$ satisfying, for any $1 \leq k_0 \leq n_0$:
\begin{itemize}
    \item a) $\eta_{k_0} \in L$ and $\mu+\eta_{k_0} \in L$;
    \item b) $w_0^{(k_0)}\big(\pi_{k_0}(\eta_{k_0})\big)=-\pi_{k_0}(\eta_{k_0})$ and $\pi_{k}(\eta_{k_0})$ is a positive integral multiple of $\pi_{k}(\eta_{k_0-1})$ for $k \in \Xi_0 \setminus \{k_0\}$;
    \item c) For any $k_0+1 \leq k \leq n_0$, $\eta_{k_0}$ is $k$-regular.
\end{itemize}
\end{lemma}
\begin{proof}
Construct the sequence by induction. Since $\eta_0=\omega_{\lambda}$, it is dominant in $L$ with component support $\Xi_0$. We know $\omega_{\lambda}$ is $1$-regular and $\mu+\omega_{\lambda}\in L$. Then by applying Lemma \ref{a.2} to $k_0=1$ and $\omega=\omega_{\lambda}$, one can obtain a dominant weight $\eta_1$ in $L$ with component support $\Xi_0$. Directly by Lemma \ref{a.2}, $\eta_1$ satisfies condition a) and b). Moreover, for any $2\leq k\leq n_0$, $\pi_{k}(\eta_1)$ is a positive integral multiple of $\pi_{k}(\omega_\lambda)$, which is $k$-regular. Thus $\eta_1$ is $k$-regular for any $2\leq k\leq k_0$. Then we have constructed $\eta_1$ satisfying all conditions.

Suppose we have constructed $\{\eta_1,...,\eta_l\}$ satisfying all conditions. By induction hypothesis, we have $\eta_l$ and $\mu+\eta_l$ are in $L$ with component support $\Xi_0$. We also have $\eta_l$ is $(l+1)$-regular. Therefore, we can apply Lemma \ref{a.2} to $k_0=l+1$ and $\omega=\eta_l$. Then we obtain a dominant weight $\eta_{l+1}$ in $L$ with component support $\Xi_0$. Again directly by Lemma \ref{a.2}, $\eta_{l+1}$ satisfies condition a) and b). Moreover, for any $l+2\leq k\leq n_0$, $\pi_{k}(\eta_{l+1})$ is a positive integral multiple of $\pi_{k}(\eta_{l})$, which is $k$-regular. Then $\eta_{l+1}$ also satisfies condition c). Then we have constructed $\eta_{l+1}$ satisfying all conditions. By induction, this lemma is proved.
\end{proof}
Now we consider the dominant weight $\eta_{n_0}\in L$ in Lemma \ref{a.3}. We know $\mu+\eta_{n_0}$ is also in $L$. For any $1\leq k \leq n_0$, we have $w_0^{(k)}\big(\pi_{k}(\eta_{n_0})\big)=-\pi_{k}(\eta_{n_0})$ since $\pi_{k}(\eta_{n_0})$ is a positive integral multiple of $\pi_{k}(\eta_{k})$ and $w_0^{(k)}\big(\pi_{k}(\eta_{k})\big)=-\pi_{k}(\eta_{k})$. Then $w_0(\eta_{n_0})=-\eta_{n_0}$ and $\eta_{n_0}$ is a desirable dominant weight $\eta\in L$ in Subsection \ref{4.a}. Therefore, the proof of Proposition \ref{3.5} reduces to the proof of Lemma \ref{a.1}.

\subsection{Proof of Lemma \ref{a.1}} \label{4.c}
To prove Lemma \ref{a.1}, we compute it depending on type since $w_0$ acts differently in different types. There are four cases in total and the following computations are base on some basic facts of Dynkin diagrams and root data (see e.g. \cite[\S 6.4]{B}).

\subsubsection{Type $A_1,B_n,C_n,D_{2n},E_7,E_8,F_4,G_2$} 

In these types, we know $w_0=-1$. Thus the sequence can be chosen as $\{\omega\}$.

\subsubsection{Type $E_6$} 
Label the vertices of the Dynkin diagram of $E_6$ as following 
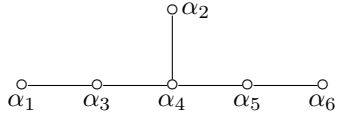
\begin{center}
    \begin{tikzpicture}[baseline=0]
    \node[] at (0,0) {$\circ$};
    \node[below] at (0,0) {$\alpha_1$};
    \draw[-] (0.08,0) to (0.93,0);
    \node[] at (1,0) {$\circ$};
    \node[below] at (1,0) {$\alpha_3$};
    \draw[-] (1.08,0) to (1.93,0);
    \node[] at (2,0) {$\circ$};
    \node[below] at (2,0) {$\alpha_4$};
    \draw[-] (2.08,0) to (2.93,0);
    \node[] at (3,0) {$\circ$};
    \node[below] at (3,0) {$\alpha_5$};
    \draw[-] (3.08,0) to (3.93,0);
    \node[] at (4,0) {$\circ$};
    \node[below] at (4,0) {$\alpha_6$};
    \draw[-] (2,0.08) to (2,0.93);
    \node[] at (2,1) {$\circ$};
    \node[right] at (2,1) {$\alpha_2$};
    \end{tikzpicture}.
\end{center}
For convenience, we denote weights in the following way. Let $\nu$ be any weight. It is a $\mathbb{Q}$-combination of simple roots, i.e., $\nu= \sum\limits_{i=1}^6 k_i \alpha_i$. Denote $\nu$ by the 6-tuple
$(\begin{array}{r@{}c@{}l} &\;\;k_2& \\
        k_1,k_3&,k_4&,k_5,k_6
   \end{array})$. Therefore, one can write down the list of fundamental weights as 6-tuples:
\begin{align*}
   \omega_1 &=\frac{1}{3}
   (\begin{array}{r@{}c@{}l}
      &\;\;3& \\
         4,5&,6&,4,2
   \end{array}),  & 
   \omega_2&=
   (\begin{array}{r@{}c@{}l}
      &\;\;2& \\
         1,2&,3&,2,1 
   \end{array}), &  \omega_3&=\frac{1}{3}(\begin{array}{r@{}c@{}l}
      &\;\;6& \\
        5,10&,12&,8,4
   \end{array}), \\ \omega_4 &=(\begin{array}{r@{}c@{}l}
      &\;\;3& \\
         2,4&,6&,4,2
   \end{array}),  & \omega_5&=\frac{1}{3}(\begin{array}{r@{}c@{}l}
      &\;\;6& \\
        4,8&,12&,10,5
   \end{array}), & \omega_6&=\frac{1}{3}(\begin{array}{r@{}c@{}l}
      &\;\;3& \\
        2,4&,6&,5,4
   \end{array}).   
 \end{align*}

In this type, we have $w_0$ transforms $\alpha_1,\alpha_2,\alpha_3,\alpha_4,\alpha_5,\alpha_6$ into $-\alpha_6,-\alpha_2,-\alpha_5,-\alpha_4,-\alpha_3,-\alpha_1,$ respectively. Then we have $w_0(\omega_2)=-\omega_2$ and $w_0(\omega_4)=-\omega_4$. Therefore, we want the dominant weight $\nu_r$ to be an nonnegative linear combination of $\omega_2$ and $\omega_4$.

Now we give the construction of the sequence $\{\nu_1,...,\nu_r\}$. First we set $\nu_1:=\omega+s_6(\omega)$. For any weight $\nu=(\begin{array}{r@{}c@{}l}
      &\;\;k_2& \\
         k_1,k_3&,k_4&,k_5,k_6
   \end{array})$, we have $s_6(\nu)=(\begin{array}{r@{}c@{}l}
      &\;\;k_2&\\
         k_1,k_3&,k_4&,k_5,k_5-k_6
   \end{array})$. Thus for fundamental weights, we have
\begin{center}
    $\left\{\begin{array}{l}
      \omega_6 +s_6(\omega_6)=\omega_5, \\
\omega_i+s_6(\omega_i)=2\omega_i ,\quad 1\leq i \leq 5.
    \end{array}\right.$
\end{center}
Therefore, we have $\nu_1\in \sum\limits_{i=1}^5 \mathbb{Z}_{\geq 0} \omega_i$ is dominant. Then $\nu_1=\omega+s_6(\omega) \in X(\omega, \omega) \subset L$ by Theorem \ref{2.3}. 
  
Then we set $\nu_2:= \nu_1 + s_5(\nu_1)$. For any weight $\nu=(\begin{array}{c}
      k_2 \\
         k_1,k_3,k_4,k_5,k_6
   \end{array})$, we have $s_5(\nu)=(\begin{array}{r@{}c@{}l}
      &\;\;k_2& \\
         k_1,k_3&,k_4&,k_4+k_6-k_5,k_6
   \end{array})$. Thus for fundamental weights, we have
\begin{center}
    $\left\{\begin{array}{l}
         \omega_5+s_5(\omega_5)=\omega_4+\omega_6, \\
    \omega_i+s_5(\omega_i)=2\omega_i ,\quad i\neq 5 .
    \end{array}\right.$
\end{center}   
Therefore, $\nu_2$ is dominant by above computation. Then we have $\nu_2=\nu_1+s_5(\nu_1)\in X(\nu_1,\nu_1) \subset L$ by Theorem \ref{2.3}. 
 
Then set $\nu_3:=\nu_2+s_6s_5(\nu_1)=\nu_1+s_5(\nu_1)+s_6s_5(\nu_1)$. By above computations, we have 
\begin{center}
    $s_6s_5(\omega_i)=\left\{ \begin{array}{ll}
      s_6(\omega_4+\omega_6-\omega_5)=\omega_4-\omega_6,   & i=5, \\
      s_6(\omega_6)=\omega_5-\omega_6,   & i=6, \\
      s_6(\omega_i)=\omega_i,  & i=1,2,3,4.
    \end{array}\right.$
\end{center}
Then we have
 \begin{center}
       $ \left\{\begin{array}{l}
         \omega_5+s_5(\omega_5)+s_6s_5(\omega_5)=\omega_4+\omega_6+\omega_4-\omega_6=2\omega_4    ,  \\
     \omega_i+s_5(\omega_i)+s_6s_5(\omega_i)=3\omega_i , \quad 1\leq i\leq 4.
         \end{array}\right. $
 \end{center}
Since $\nu_1 \in \sum\limits_{i=1}^5 \mathbb{Z}_{\geq 0} \omega_i$, we have $\nu_3 \in \sum\limits_{i=1}^4 \mathbb{Z}_{\geq 0} \omega_i$ is dominant. Still by Theorem \ref{2.3}, we have $\nu_3=\nu_2+s_6 s_5(\nu_1) \in X(\nu_2,\nu_1) \subset L$. 
  
Now we have obtained $\nu_3 \in L$ and $\nu_3 \in \sum\limits_{i=1}^4 \mathbb{Z}_{\geq 0} \omega_i$. Actually, the coefficients of $\omega_5$ and $\omega_6$ vanished in the process of obtaining $\nu_3$. By symmetry of the Dynkin diagram of type $E_6$ and the simple roots, we can make the coefficient of $\omega_1$ and $\omega_3$ vanished in a similar way. Set $\nu_4:=\nu_3+s_1(\nu_3)$, one can check $\nu_4 \in \sum\limits_{i=2}^4 \mathbb{Z}_{\geq 0} \omega_i$ is dominant and is in $L$ by Theorem \ref{2.3}. Then we set $\nu_5=\nu_4+s_3(\nu_4)$ and $\nu_6=\nu_5+s_1s_3(\nu_4)=\nu_4+s_3(\nu_4)+s_1s_3(\nu_4)$. Similar to $\nu_2$ and $\nu_3$, we have $\nu_5$ and $\nu_6$ are in $L$ and $\nu_6 \in \mathbb{Z}_{\geq 0} \omega_2 +\mathbb{Z}_{\geq 0} \omega_4$. Then we have $w_0(\nu_6)=-\nu_6$ by our computations before. Therefore, $\{\nu_1,...,\nu_6\}$ is a desirable sequence for Lemma \ref{a.1} in type $E_6$.

\subsubsection{Type $A_n(n\geq 2)$} 
Label the vertices of the Dynkin diagram of $A_n$ as following 
\begin{center}
   \begin{tikzpicture}[baseline=0]
   \node[] at (0,0) {$\circ$};
   \node[below] at (0,0) {$\alpha_1$};
   \draw[-] (0.08,0) to (0.93,0);
   \node[] at (1,0) {$\circ$};
   \node[below] at (1,0) {$\alpha_2$};
   \draw[dashed] (1.08,0) to (2.92,0);
   \node[] at (2.95,0) {$\circ$};
   \node[below] at (2.95,0) {$\alpha_{n-1}$};
   \draw[-] (3.03,0) to (3.88,0);
   \node[] at (3.95,0) {$\circ$};
   \node[below] at (3.95,0) {$\alpha_n$};
   \end{tikzpicture}.
\end{center}
For convenience, we still denote each weight $\nu$ by a $n-$tuple $(k_1,k_2,...,k_n)$, where $\nu=\sum\limits_{i=1}^n k_i \alpha_i$. Then for any $1\leq i\leq n$, the fundamental weight $\omega_i$ is denoted by
 \begin{center}
     $\omega_i=\frac{1}{n+1}(n-i+1,2(n-i+1),...,i(n-i+1),i(n-i),...,2i,i)$.
 \end{center}
In this type, $w_0$ transforms $\alpha_i$ into $-\alpha_{n+1-i}$. Then we have $w_0(\omega_1)=-\omega_n$ and $w_0(\omega_n)=-\omega_1$. Therefore, we want the dominant weight $\nu_r$ to be an nonnegative linear combination of $\omega_1$ and $\omega_n$. 

Now we give the construction of the sequence $\{\nu_1,...,\nu_r\}$. First we claim that there is a sequence $\{\zeta_j\}_{i=1}^n$ of dominant weights in $L$ such that, for any $1 \leq i \leq n$, $\zeta_i \in \sum\limits_{l=i}^n \mathbb{Z}_{> 0}\omega_l$.
Symmetrically, there is also a sequence $\{\theta_i\}_{i=1}^n$ of nonzero dominant weights in $L$ such that, for any $1\leq i\leq n$, $\theta_i \in \sum\limits_{l=1}^{n+1-i} \mathbb{Z}_{> 0}\omega_l$.

We proceed to prove the claim. Construct the sequence $\{\zeta_i\}_{i=1}^n$ as following:
\newline
(a) $\zeta_1=\omega$;
\newline
(b) $\zeta_{i+1}=\zeta_i+s_i(\zeta_i)+s_{i-1} s_i(\zeta_i)+...+s_1 s_2...s_{i-1} s_i (\zeta_i)$. For any $1\leq m\leq i+1$, denote the sum of first $m$ terms of right hand side by $\zeta_{i+1,m}$.

Symmetrically, we construct the sequence $\{\theta_i\}_{i=1}^n$ as following:
\newline
(a) $\theta_1=\omega$;
\newline
(b) $\theta_{i+1}=\theta_i+s_{n+1-i}(\theta_i)+s_{n+1-(i-1)} s_{n+1-i}(\theta_i)+...+s_n s_{n-1}...s_{n+1-(i-1)} s_{n+1-i} (\theta_i)$. For any $1\leq m\leq i+1$, denote the sum of first $m$ terms of right hand side by $\theta_{i+1,m}$.

Then we check these sequences satisfy our requirements by induction on $i$. By symmetry of the Dynkin diagram of type $A_n$ and the simple roots, we only need to check for $\{\zeta_i\}_{i=1}^n$. For any weight $\nu=(k_1,k_2,...,k_n)$, we have
\begin{center}
    $s_i(\nu)=\left\{ \begin{array}{ll}
      (k_1,k_2,...,k_{i-1},k_{i-1}+k_{i+1}-k_i,k_{i+1},...,k_n),  & i\neq 1,n, \\
      (k_2-k_1,k_2,...,k_{i-1},k_i,k_{i+1},...,k_n),   & i=1, \\
      (k_1,k_2,...,k_{i-1},k_i,k_{i+1},...,k_{n-1}-k_n),  & i=n.
    \end{array}\right.$
\end{center}
Thus for fundamental weights, we have
\begin{center}
    $s_i(\omega_l)=\left\{ \begin{array}{ll}
      \omega_{l-1}+\omega_{l+1}-\omega_l,  &  2 \leq l=i \leq n-1, \\
      \omega_2-\omega_1,  &  l=i=1, \\
      \omega_{n-1}-\omega_n,  & l=i=n, \\
      \omega_l, & l\neq i.
    \end{array}\right.$
\end{center}
For $i=1$, it is clear that $\zeta_1=\omega$ is in $L$ and $\omega\in \sum\limits_{l=1}^n \BZ_{> 0} \omega_l$. Suppose that $\{\zeta_1,...,\zeta_i\}$ satisfies our requirements where $i<n$. Then we look at $\zeta_{i+1}$. We first compute $\omega_i+s_i(\omega_i)+s_{i-1} s_i(\omega_i)+...+s_1 s_2...s_{i-1} s_i (\omega_i)$. For this we need the following lemma. Set $\omega_0=0$ for convenience.

\begin{lemma}\label{a.4}
Suppose that $1 \leq i \leq n-1$. Then $s_{i-l+1}... s_{i-1} s_i(\omega_i)=\omega_{i-l}-\omega_{i-l+1}+\omega_{i+1}$ for $1 \leq l \leq i$.
\end{lemma}

\begin{proof}
Proceed by induction on $l$. For $l=1$, we have $s_i(\omega_i)=\omega_{i-1}-\omega_i+\omega_{i+1}$ by our computations above. Suppose that the equation holds for $l-1$. Then for $l$, we have
\begin{align*}
    s_{i-l+1}...s_{i-1} s_i(\omega_i)&=s_{i-l+1}(\omega_{i-l+1}-\omega_{i-l+2}+\omega_{i+1}) \\
&=\omega_{i-l}-\omega_{i-l+1}+\omega_{i-l+2}-\omega_{i-l+2}+\omega_{i+1} \\
&=\omega_{i-l}-\omega_{i-l+1}+\omega_{i+1}.
\end{align*}
Therefore, the equation also holds for $l$. By induction, the lemma is proved.
\end{proof}

We know $\zeta_i\in \sum\limits_{l=i}^n \BZ_{> 0} \omega_l$ by induction hypothesis. Now by Lemma $\ref{a.4}$, we have $\omega_i+s_i(\omega_i)+s_{i-1} s_i(\omega_i)+...+s_1 s_2...s_{i-1} s_i (\omega_i)=\omega_i+(\omega_{i-1}-\omega_i+\omega_{i+1})+(\omega_{i-2}-\omega_{i-1}+\omega_{i+1})+...+(\omega_0-\omega_1+\omega_{i+1})\in \mathbb{Z}_{> 0}\omega_{i+1}$. This equation together with the fact that $s_1,s_2,...,s_i$ fix $\omega_l$ for $i+1\leq l\leq n$ show that $\zeta_{i+1}$ is contained in $\sum\limits_{l=i+1}^n \BZ_{> 0} \omega_l$.

Then we show that $\zeta_{i+1}$ is in $L$. Recall our construction of $\{\zeta_i\}_{i=1}^n$ and $\{\zeta_{i+1,m}\}_{m=1}^{i+1}$. Again by Lemma \ref{a.4}, for any $1 \leq m \leq i+1$, the sum of the first $m$ terms of $\omega_i+s_i(\omega_i)+s_{i-1} s_i(\omega_i)+...+s_1 s_2...s_{i-1} s_i (\omega_i)$ is $\omega_i+(\omega_{i-1}-\omega_i+\omega_{i+1})+...+(\omega_{i-m+1}-\omega_{i-m+2}+\omega_{i+1}) 
=(m-1)\omega_{i+1}+\omega_{i-m+1}$, which is dominant. Still together by the fact that $s_1,s_2,...,s_i$ fix $\omega_l$ for $i+1\leq l\leq n$, $\zeta_{i+1,m}$ is always dominant for $1 \leq m \leq i+1$. Then for any $2\leq m\leq i+1$, we have $\zeta_{i+1,m}\in X(\zeta_{i+1,m-1},\zeta_i)$. Since $\zeta_{i+1,1}=\zeta_i$ is in $L$ by induction hypothesis, we have $\zeta_{i+1,m}$ are in $L$ for all $1\leq m\leq i+1$ by applying Theorem \ref{2.3} successively. In particular, $\zeta_{i+1}=\zeta_{i+1,i+1}$ is in $L$. Therefore, the claim is true by induction and symmetry. 

Now that the claim is true. Then we have $\zeta_n=a\omega_n\in L$ and $\theta_n=b\omega_1\in L$ for some positive integers $a,b$. We set $\zeta=b\zeta_n+a\theta_n \in L$. Notice that $w_0(\zeta)=ab w_0(\omega_1+\omega_n)=-\zeta$.

By above construction, we have two sequences $\{\zeta_{i,m}\mid 2\leq i \leq n, 2\leq m \leq i\}$ and $\{\theta_{i,m}\mid 2\leq i \leq n, 2\leq m \leq i\}$ in $L$. We know that
\begin{center}
    $\left\{ \begin{array}{l}
      \zeta_{2,2}=\omega+s_1(\omega),  \\
      \theta_{2,2}=\omega+s_n(\omega),  \\
      \zeta_{i+1,m}=\zeta_{i+1,m-1}+ s_{i-m+2}...s_i (\zeta_{i,i}), \quad i\geq 2,2<m\leq i+1, \\
      \zeta_{i+1,2}=\zeta_{i,i}+s_i(\zeta_{i,i}), \quad i\geq 2, \\
      \theta_{i+1,m}=\theta_{i+1,m-1}+ s_{n+1-i+(m-2)}...s_{n+1-i} (\theta_{i,i}), \quad i\geq 2,2<m\leq i+1, \\
      \theta_{i+1,2}=\theta_{i,i}+s_{n+1-i}(\theta_{i,i}), \quad i\geq 2,  \\
      l\zeta_{n,n}=(l-1)\zeta_{n,n}+\zeta_{n,n}, \quad l\geq 2, \\
      l\theta_{n,n}=(l-1)\theta_{n,n}+\theta_{n,n}, \quad l\geq 2, \\
      \zeta=b\zeta_{n,n}+a\theta_{n,n}.
    \end{array}\right.$
\end{center}
Then consider the following sequence in $L$,
\begin{center}
    $\{\omega, \zeta_{2,2},\zeta_{3,2},\zeta_{3,3},...,\zeta_{n,n}, 2\zeta_{n,n},..., b\zeta_{n,n}, \theta_{2,2},\theta_{3,2},\theta_{3,3},...,\theta_{n,n}, 2\theta_{n,n},...,a\theta_{n,n},\zeta\} $.
\end{center}
This is a desirable sequence for Lemma \ref{a.1} in type $A_n$ by equations above.

\subsubsection{Type $D_{2n+1}$} 
Label the vertices of the Dynkin diagram of $D_{2n+1}$ as following
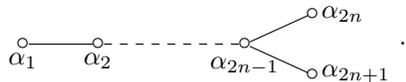
\begin{center}
     \begin{tikzpicture}[baseline=0]
     \node[] at (0,0) {$\circ$};
     \node[below] at (0,0) {$\alpha_1$};
     \draw[-] (0.08,0) to (0.93,0);
     \node[] at (1,0) {$\circ$};
     \node[below] at (1,0) {$\alpha_2$};
     \draw[dashed] (1.08,0) to (2.92,0);
     \node[] at (2.95,0) {$\circ$};
     \node[below] at (2.95,-0.05) {$\alpha_{2n-1}$};
     \draw[-] (3.02,0.03) to (3.79,0.38);
     \node[] at (3.85,0.4) {$\circ$};
     \node[right] at (3.85,0.4) {$\alpha_{2n}$};
     \draw[-] (3.02,-0.03) to (3.79,-0.38);
     \node[] at (3.85,-0.4) {$\circ$};
     \node[right] at (3.85,-0.4) {$\alpha_{2n+1}$};
     \end{tikzpicture}.
\end{center}
For convenience, we still denote each weight $\nu=\sum\limits_{i=1}^{2n+1}k_i \alpha_i$ by a $(2n+1)$-tuple $(k_1,k_2,...,k_{2n},k_{2n+1})$. Then we have the list of fundamental weights as $(2n+1)$-tuples
\begin{align*}
    &\omega_{2n}=(\frac{1}{2},\frac{2}{2}=1,...,\frac{2n-1}{2},\frac{2n+1}{4},\frac{2n-1}{4}), \\
    &\omega_{2n+1}=(\frac{1}{2},\frac{2}{2}=1,...,\frac{2n-1}{2},\frac{2n-1}{4},\frac{2n+1}{4}), \\
    &\omega_i=(1,2,...,i-1,i,i,...,i,\frac{i}{2},\frac{i}{2}), \quad 1\leq i\leq 2n-1.
\end{align*}
In this type, we have $w_0$ transforms $\alpha_{2n}, \alpha_{2n+1}$ into $-\alpha_{2n+1}, -\alpha_{2n}$ and acts as $-1$ on other simple roots. Then we have $w_0(\omega_i)=-\omega_i$ for $1\leq i \leq 2n-1$. Therefore,  we want the dominant weight $\nu_r$ to lie in $\sum\limits_{i=1}^{2n-1} \mathbb{Z}_{\geq 0} \omega_i$. 

Now we give the construction of the sequence $\{\nu_1,...,\nu_r\}$. First set $\nu_1:=\omega+s_{2n+1}(\omega)$. For any weight $\nu=(k_1,k_2,...,k_{2n+1})$, we have $s_{2n+1}(\nu)=(k_1,k_2,...,k_{2n-1},k_{2n},k_{2n-1}-k_{2n+1})$. Thus for fundamental weights, we have
\begin{center}
    $\left\{\begin{array}{l}
           \omega_{2n+1}+s_{2n+1}(\omega_{2n+1})=\omega_{2n-1}, \\
    \omega_{i}+s_{2n+1}(\omega_{i})=2\omega_{i}, \quad i \neq 2n+1.
    \end{array}\right.$
\end{center}
Therefore, we have $\nu_1 \in \sum\limits_{i=1}^{2n} \mathbb{Z}_{\geq 0} \omega_i$ is dominant. By Theorem \ref{2.3}, we have $\nu_1=\omega+s_{2n+1}(\omega)\in X(\omega,\omega)\subset L$.

Then we set $\nu_2:=\nu_1+s_{2n}(\nu_1)$. For any weight $\nu=(k_1,k_2,...,k_{2n+1})$, we have $s_{2n}(\nu)=(k_1,k_2,...,k_{2n-1},k_{2n-1}-k_{2n},k_{2n+1})$. Thus for fundamental weights, we have
\begin{center}
    $\left\{\begin{array}{l}
       \omega_{2n}+s_{2n}(\omega_{2n})=\omega_{2n-1}, \\
    \omega_{i}+s_{2n}(\omega_{i})=2\omega_{i}, \quad i\neq 2n.
    \end{array}\right.$
\end{center}
Since $\nu_1 \in \sum\limits_{i=1}^{2n} \mathbb{Z}_{\geq 0} \omega_i$, we have $\nu_2 \in \sum\limits_{i=1}^{2n-1} \mathbb{Z}_{\geq 0} \omega_i$ is dominant. By Theorem \ref{2.3} we have $\nu_2=\nu_1+s_{2n}(\nu_1)\in X(\nu_1,\nu_1)\subset L$. Since $\nu_2 \in \sum\limits_{i=1}^{2n-1} \mathbb{Z}_{\geq 0} \omega_i$, we have $w_0(\nu_2)=-\nu_2$ by our computations before. Therefore, we have $\{\nu_1,\nu_2\}$ is a desirable sequence for Lemma \ref{a.1} in type $D_{2n+1}$.

\section{Reductive case and Comparison}
\subsection{Comparison with Vinberg's results}
The definition of perfect submonoids was given by Vinberg in \cite[\S 1]{V}. Vinberg used this definition to develope his classification of reductive algebraic monoids. All algebraic monoids in this subection are assumed to be linear and irreducible.

Let $G$ be a connected reductive algebraic group with a maximal torus $T$. The natural action of $G\times G$ on $K[G]$ induces \cite[II.3.1 Satz 3]{KW85}
\begin{center}
    $K[G]=\bigoplus\limits_{\lambda \in X^*_+(T)}K[G]_\lambda$,
\end{center}
where $K[G]_{\lambda}\simeq L(\lambda)^*\bigotimes L(\lambda)$ is the linear space spanned by matrix entries of $L(\lambda)$. Every $(G\times G)$-stable subspace of $K[G]$ has the form of $K[G]_L=\bigoplus\limits_{\lambda \in L}K[G]_{\lambda}$ for some subset $L$ of $X^*_+(T)$. Let $M$ be a reductive monoid with unit group $G$. Then $K[M]$ is a $G\times G$-stable subalgebra of $K[G]$. By checking the multiplication of $K[M]$, we have $K[M]=K[G]_L$ where $L$ is a perfect submonoid of $X^*_+(T)$. Vinberg gave a description of reductive monoids with unit group $G$.

\begin{theorem}\cite[Theorem 1]{V}\label{5.1}
A submonoid $L$ of $X^*_+(T)$ defines an algebraic monoid $M$ with unit group $G$, if and only if $L$ is perfect, finitely generated and generating $X^*(T)$ as a group.
\end{theorem}

Moreover, based on the fact that every algebraic monoid admits a normalization \cite[Proposition 3.15]{Re}, Vinberg gave a classification of normal reductive monoids in \cite[Theorem 2]{V}. This classification is important in Vinberg's construction of Vinberg monoids in \cite[Theorem 5]{V}.

By \cite[Lemma 1.1]{KKDS}, if $L$ is perfect, finitely generated and generates $X^*(T)$ as a group, then the reductive monoid $M$ defined by $L$ is normal if and only if $L\subset X^*(T)$ is \textit{saturated} in the following sense.
\begin{definition}\cite[Definition 1.2]{KKDS}\label{5.3}
Let $L$ be a subset of $X^*(T)$. Suppose that for any $ \lambda \in X^*(T)$, if there is an integer $n>1$ such that $n\lambda \in L$, then $\lambda\in L$. Then $L$ is called \textit{saturated}.
\end{definition}
\begin{remark}
Note this definition is different from the definition of \textit{saturated} in Section \ref{2}. For example, let $\lambda$ be a dominant weight. The subset $\Pi(2\lambda)\subset X^*(T)$ is not always saturated in the sense of Definition \ref{5.3} since $\lambda$ may not be in $\Pi(2\lambda)$. In this section we are always using the definition from \cite{KKDS}.
\end{remark}
Therefore, Vinberg's results give a characterization for perfect submonoids of $X^*_+(T)$ which are finitely generated, saturated and generates $X^*(T)$ as a group. Then we compare our results with Vinberg's results on perfect submonoids of dominant weights.

\subsubsection{Semisimple case}
Suppose that $G$ is semisimple. Let $G^{sc}$ be the simply connected cover of $G$ with a maximal torus $T^{sc}$, and $\Xi=\{1,2,...,n\}$ be the index set of the quasi-simple factors of $G^{sc}$.

\begin{lemma}\label{5.5}
Suppose that $G$ is a connected semisimple group. The perfect submonoids of $X^*_+(T)$ are all finitely generated.
\end{lemma}

\begin{proof}
Recall the notations in Subsection \ref{3.e}. By Subsection \ref{3.e}, every perfect submonoid $L$ of $X^*_+(T)$ with component support $\Xi_0 \subset \Xi$ can be viewed as a perfect submonoid of $X^*_+(T^{sc})_{\Xi_0}$ with full component support. Therefore, we may assume that $G$ is simply connected and $L$ has full component support. Then $L=\Tilde{L}\cap X^*_+(T)$, where $\Tilde{L}$ is a sublattice of $X^*(T)$ containing $Q$, by Theorem \ref{A}. Therefore, we have $L$ is the set of dominant weights of a connected semisimple group $G'$ with simply connected cover $G$. Thus $L$ is clearly finitely generated.
\end{proof}

Therefore, our result differs from Vinberg's theorem in the sense that, we do not assume that $L$ is saturated or generates $X^*(T)$ as a group. Actually, these two conditions do not hold in general. In conclusion, our results give a complete characterization of all perfect submonoids of dominant weights in semisimple case.

\subsubsection{Reductive case}
Suppose that $G$ is reductive but not semisimple. Let $G_0$ be its derived subgroup and $Z^0$ be its connected center. Fix a maximal unipotent subgroup $U$ and a Borel subgroup $B=TU$ of $G$. Let $T_0=T\cap G_0$ be a maximal torus of $G_0$. Since $T$ is an almost direct product of $T_0$ and $Z^0$, there is an natural embedding
\begin{center}
    $i:X^*(T)\to X^*(T_0)\bigoplus X^*(Z^0)$
\end{center}
given by restrictions. Therefore, we identify each dominant weight $\lambda\in X^*_+(T)$ with a pair $i(\lambda)=(\mu,\nu)$, where $\mu=\lambda|_{T_0}\in X^*_+(T_0)$ and $\nu=\lambda|_{Z^0}\in X^*_+(Z^0)$. Then we naturally relate perfect submonoids of $X^*_+(T)$ to perfect submonoids of $X^*_+(T_0)\bigoplus X^*_+(Z^0)$.

\begin{proposition}\label{5.6}
The perfect submonoids of $X^*_+(T)$ are exactly the perfect submonoids of $X^*_+(T_0)\bigoplus X^*_+(Z^0)$ contained in the image of $i:X^*(T)\to X^*(T_0)\bigoplus X^*(Z^0)$.
\end{proposition}
\begin{proof}
It is clear that the embedding $i$ gives a 1-1 correspondence between submonoids of $X^*_+(T)$ with submonoids of $X^*_+(T_0)\bigoplus X^*_+(Z^0)$ contained in $\Im(i)$. It remains to show that $L\subset X^*_+(T)$ is perfect if and only if $i(L)\subset X^*_+(T_0)\bigoplus X^*_+(Z^0)$ is perfect.

Let $(\mu_1,\nu_1)=i(\lambda_1)$ and $(\mu_2,\nu_2)=i(\lambda_2)$ be arbitrary. We have 
\begin{center}
    $X\big((\mu_1,\nu_1),(\mu_2,\nu_2)\big)=\{(\mu,\nu_1+\nu_2)\mid \mu\in X(\mu_1,\mu_2)\subset X^*_+(T_0)\}$.
\end{center}
By Lemma \ref{2.1} and the fact that the simple roots of $G$ are exactly the simple roots of $G_0$, we have
\begin{center}
    $i\big(X(\lambda_1,\lambda_2)\big)=\{(\lambda|_{T_0},\nu_1+\nu_2)\mid \lambda\in X(\lambda_1,\lambda_2)\}$.
\end{center}

Now we consider the tensor product decomposition
\begin{center}
    $L(\lambda_1)\bigotimes L(\lambda_2)=\bigoplus\limits_{\lambda \in X(\lambda_1,\lambda_2)} L(\lambda)^{\oplus m_{\lambda_1,\lambda_2}^\lambda}$.
\end{center}
We know that $L(\lambda)=\{f\in K[G]\mid f(tg)=\lambda(t)f(g),\forall t\in T,g\in G\}$. By restricting on $G_0$, we have $L(\lambda)|_{G_0}$ is clearly a nontrivial $G_0$-module. Let $i(\lambda)=(\mu,\nu)$, which means $\lambda|_{T_0}=\mu$. Then by decomposing $L(\lambda)|_{G_0}$ into a direct sum of irreducible $G_0$-modules, we have $L(\lambda)=L_0(\mu)^{\oplus a_\lambda}$ for some positive integer $a_\lambda$, where $L_0(\mu)$ denotes the irreducible $G_0$-module with highest weight $\mu$. Therefore, the restriction of above equation on $G_0$ gives 
\begin{center}
    $L_0(\mu_1)^{\oplus a_{\lambda_1}}\bigotimes L_0(\mu_2)^{\oplus a_{\lambda_2}}=\bigoplus\limits_{\mu=\lambda|_{T_0},\lambda\in X(\lambda_1,\lambda_2)} L_0(\mu)^{\oplus a_\lambda m_{\mu_1,\mu_2}^\mu }$.
\end{center}
Recall the tensor product decomposition of $G_0$-modules
\begin{center}
    $L_0(\mu_1)\bigotimes L_0(\mu_2)=\bigoplus\limits_{\mu \in X(\mu_1,\mu_2)} L_0(\mu)^{\oplus m_{\mu_1,\mu_2}^\mu}$.    
\end{center}
By comparing the direct summands of the right hand side of two equations, we have $\{\lambda|_{T_0}\mid \lambda\in X(\lambda_1,\lambda_2)\}=X(\mu_1,\mu_2)$. Then $X\big((\mu_1,\nu_1),(\mu_2,\nu_2)\big)=i\big(X(\lambda_1,\lambda_2)\big)$. Then $X\big((\mu_1,\nu_1),(\mu_2,\nu_2)\big)\subset i(L)$ is equivalent to $X(\lambda_1,\lambda_2)\subset L$. Therefore, we have $L\subset X^*_+(T)$ is perfect if and only if $i(L)\subset X^*_+(T_0)\bigoplus X^*_+(Z^0)$ is perfect and the proposition is proved.
\end{proof}

Now Let $L$ be a submonoid of $X^*_+(T_0)\bigoplus X^*_+(Z^0)$ and $pr_Z(L)$ be the projection of $L$ to $X^*_+(Z^0)$. One can write $L$ as
\begin{center}
    $L=\bigcup\limits_{\nu\in pr_Z(L)} \{(\mu,\nu) \mid \mu \in L_{\nu}\}$,
\end{center}
where $L_{\nu}:=\{\mu \in X^*_+(T_0) \mid (\mu,\nu)\in L\}$. In general, the perfect submonoids of $X^*_+(T_0)\bigoplus X^*_+(Z^0)$ could be complicated when $G$ is reductive but not semisimple. To see that, we consider the following example when $L_0=\{0\}$.

\begin{example}\label{5.7}
Construct the perfect submonoid $L$ of $X^*_+(T_0)\bigoplus X^*_+(Z^0)$ as following.

Let $pr_Z(L)$ be $\BZ_{\geq 0}\nu$, where $\nu$ is nonzero in $X^*_+(Z^0)$. We construct $L_{i\nu}$ $(i\in \BZ_{\geq 0}) $ inductively. For $i=0$, let $L_{0}=\{0\}$. Suppose we have already constructed $L_{i\nu}$ for $i<m$. For $i=m$, we take an arbitrary subet $X_m$ of $X^*_+(T)$ and construct $L_{m\nu}$ as
\begin{center}
    $L_{m\nu}=\big(\bigcup\limits_{k=1}^{m-1} X(L_{k\nu},L_{(m-k)\nu})\big)\bigcup X_{m}$,
\end{center}
where $X(L_{k\nu},L_{(m-k)\nu})$ denotes the union of sets $X(\lambda,\mu)$ for all $\lambda\in L_{k\nu}$ and $\mu \in L_{(m-k)\nu}$. 

We check the perfectness of $L$. Let $(\mu_1,i\nu),(\mu_2,j\nu)$ be any two dominant weights in $L$. By our discussion in Proposition \ref{5.6}, every dominant weight in $X\big((\mu_1,i\nu),(\mu_2,j\nu)\big)$ has the form of $\big(\mu,(i+j)\nu\big)$ for some $\mu \in X(\mu_1,\mu_2)\subset X(L_{i\nu},L_{j\nu})$. By our construction above, $\mu \in L_{(i+j)\nu}$ and thus $\big(\mu,(i+j)\nu\big)\in L$. Then $L$ is perfect. Since $X_m$ are all arbitrarily chosen, it is difficult to characterize such $L$.
\end{example}


\begin{thebibliography}{99}
\bibitem[1]{B} N. Bourbaki, \emph{Lie groups and Lie algebras (Chapters 4-6)}, translated from the 1968 French original by Andrew Pressley, Elements of Mathematics (Berlin), Springer-Verlag, Berlin, 2002.

\bibitem[2]{BT} A. Borel, J. Tits, \emph{Groupes réductifs} (French). Inst. Hautes \'Etudes Sci. Publ. Math. \textbf{27} (1965), 55-150.



\bibitem[3]{EGNO} P. Etingof, S. Gelaki, D. Nikshych, V. Ostrik, \emph{Tensor categories}, Mathematical Surveys and Monographs, Vol. 205. American Mathematical Society, Providence, RI, 2015.

\bibitem [4]{Hum} J. E. Humphreys, \emph{Introduction to Lie algebras and representation theory}, Graduate Texts in Mathematics, Vol. 9. Springer-Verlag, New York-Berlin, 1972.

\bibitem[5]{KKDS} G. Kempf, F. Knudsen, D. Mumford, B. Saint-Donat, \emph{Toroidal embeddings. I}. Lecture Notes in Mathematics, Vol. 339. Springer-Verlag, Berlin-New York, 1973.

\bibitem[6]{Ko} B. Kostant, \emph{A formula for the multiplicity of a weight}, Amer. Math. Soc. Transl.
\textbf{93} (1959), 53-73.

\bibitem[7]{Ku} S. Kumar, \emph{Proof of the Parthasarathy-Ranga Rao-Varadarajan conjecture}, Invent. Math. \textbf{93} (1988), 117-130.


\bibitem[8]{Ku3} S. Kumar, \emph{Tensor product decomposition}, in: Proceedings of the International Congress of Mathematicians, Hyderabad, India, 2010.

\bibitem[9]{KW85} H. Kraft, A. Wiedemann, \emph{Geometrische methoden in der invariantentheorie} (German), Aspects of Mathematics, D1. Friedr. Vieweg \& Sohn, Braunschweig, 1984.


\bibitem[10]{PRV} 
K. R. Parthasarathy, R. R. Rao, V. S. Varadarajan, \emph{Representations of complex semi-simple Lie groups and Lie algebras},
Ann. Math. \textbf{85} (1967), 383-429.


\bibitem[11]{Re} L. E. Renner, \emph{Linear algebraic monoids}, Encyclopaedia of Mathematical Sciences, Vol. 134. Invariant Theory and Algebraic Transformation Groups, V. Springer-Verlag, Berlin, 2005.




\bibitem[12]{V} 
\`E. B. Vinberg, \emph{On reductive algebraic semigroups}, in: Lie Groups and Lie Algebras, E. B. Dynkin's Seminar, Amer. Math. Soc. Transl., Series 2, \textbf{169} (1994), 145-182.



\end{thebibliography}
\end{document}